\PassOptionsToPackage{dvipsnames}{xcolor}
\documentclass[11pt,reqno,oneside]{amsart}
\usepackage{amsmath, amssymb, graphicx, mathrsfs, dsfont, hyperref}
\usepackage{epsfig}
\usepackage{setspace}
\usepackage{amsthm}
\usepackage[top=1in, bottom=1in, left=1in, right=1in]{geometry} 
\usepackage{pgfplots}
\usepackage[dvipsnames]{xcolor}
\usepackage{mathtools}
\usepackage{skull, color}
\usepackage{mathabx}
\usepackage{float}
\usepackage{cite}
\usepackage{caption}
\usepackage{adjustbox}
\usepackage{enumitem}
\usepackage{soul}

\theoremstyle{definition}
\newtheorem{theorem}{Theorem}[section]
\newtheorem{lemma}[theorem]{Lemma}
\newtheorem{corollary}[theorem]{Corollary}
\newtheorem{question}[theorem]{Question}
\newtheorem{conjecture}[theorem]{Conjecture}
\newtheorem{proposition}[theorem]{Proposition}
\newtheorem{definition}[theorem]{Definition}

\newtheorem*{theorem*}{Theorem}
\newtheorem*{conjecture*}{Conjecture}
\newtheorem*{question*}{Question}
\newtheorem*{lemma*}{Lemma}
\newtheorem*{corollary*}{Corollary}
\newtheorem*{definition*}{Definition}
\newtheorem*{example*}{Example}

\newcommand{\lcm}{\text{lcm}}

\title{Connecting Zeros in Pisano Periods to Prime Factors of $K$-Fibonacci Numbers}

\author[Brennan Benfield]{Brennan Benfield}
\address{Brennan Benfield: Department of Mathematics and Statistics, University of North Carolina at Charlotte, 9201 University City Blvd., Charlotte, NC 28223, USA}
\email{bbenfie3@charlotte.edu}

\author[Oliver Lippard]{Oliver Lippard}
\address{Oliver Lippard: Department of Mathematics and Statistics, University of North Carolina at Charlotte, 9201 University City Blvd., Charlotte, NC 28223, USA}
\email{hlippard@charlotte.edu}


\pgfplotsset{compat=1.18}

\begin{document}

\begin{abstract}
The Fibonacci sequence is periodic modulo every positive integer $m>1$, and perhaps more surprisingly, each period has exactly 1, 2, or 4 zeros that are evenly spaced, which also holds true for more general $K$-Fibonacci sequences. This paper proves several conjectures connecting the zeros in the Pisano period to the prime factors of $K$-Fibonacci numbers.  The congruence classes of indices for $K$-Fibonacci numbers that are multiples of the prime factors of $m$ completely determine the number of zeros in the Pisano period modulo $m$.
\end{abstract}

\maketitle

\section{Introduction}
The Fibonacci sequence is defined by $F_0=0$, $F_1=1$, and $F_n=F_{n-1}+F_{n-2}$. A curious property was first recognized in 1877 by Lagrange \cite{Lagrange}: the terms in the Fibonacci sequence modulo $10$ (i.e. the one's place digits) repeat every $60$ terms. Inquiry regarding the periodicity of the sequence modulo a positive integer has continued and in 2004 it was proven by Everest \& Shparlinski \cite{[Everest]} that \textit{every} binary recurrence sequence is periodic modulo a positive integer $m>1$. 

Define the \textit{Pisano period} as the length of one (shortest) period of the Fibonacci sequence modulo $m$, denoted $\pi(m)$. It is well established \cite{Gupta, Renault, Wall} that one period of the Fibonacci sequence modulo any $m>1$ has exactly 1, 2, or 4 zeros, equally spaced in the cycle. The number of zeros in a Pisano period is the \textit{order} of $m$, denoted $\omega(m)$. The \textit{rank} of $m$ is the index of the Fibonacci number of the first zero in a Pisano period. In other words, the rank, denoted $\alpha(m)$, is the index of the first Fibonacci number divisible by $m$. There are two conjectures in the OEIS: A053029 \cite{[A053029]} and A053031 \cite{[A053031]} regarding exactly which $m$ has order 1, 2, or 4. Note that the Fibonacci sequence mod 1 is traditionally not considered and $\omega(1)$ is defined to equal 1.
\begin{conjecture}[A053029 \cite{[A053029]}]\label{conj:oeis4}
    An integer $m$ has four zeros in its Pisano period if and only if $m$ is an odd number, all of whose factors have four zeros in their Pisano period, or if $m$ is twice such a number.
\end{conjecture}

\begin{conjecture}[A053031 \cite{[A053031]}]\label{conj:oeis1}
    An integer $m$ has one zero in its Pisano period if and only if $m$ is an odd number, all of whose factors have one zero in their Pisano period, or if $m$ is twice or four times such a number.
\end{conjecture}
It follows that all other positive integers have exactly two zeros in their Pisano period. This paper proves these conjectures and uses this curious property to show the relationship between the order of $m$ and the prime factors of $F_n$. Namely, the order of a positive integer depends on exactly when it is a prime factor of $F_n$ according to the congruence class of $n$.
\begin{theorem}\label{main_theorem} For a positive integer $m$,
    \begin{enumerate}[label=(\roman*)]
        \item For a positive integer $m$, $\omega(m) = 4$ for $m > 3$ if and only if m has prime factorization $m = 2^jp_1^{e_1} \cdots p_r^{e_r}$ where $j \in \{0,1\}$ and for each $1 \varleq i \varleq r$ there is an odd index $n_i$ such that $p_i \mid F_{n_i}$.\\
        \item For a positive integer $m$, $\omega(m) = 1$ for $m > 3$ if and only if m has prime factorization $m = 2^jp_1^{e_1} \cdots p_r^{e_r}$ where $j \in \{0,1,2\}$ and for each $1 \varleq i \varleq r$ there is an index $n_i \equiv 2 \pmod 4$ such that $p_i \mid F_{n_i}$, but $p_i \not \mid F_{\nu_i}$ for all odd $\nu_i$.\\
        \item $\omega(m)=2$ for all other positive integers $m$.
    \end{enumerate}
\end{theorem}
This list is exhaustive, in the sense that every natural number falls into exactly one category of Theorem \ref{main_theorem}. Note that every prime divides some Fibonacci number, and with the exceptions of $n=1, 2, 6, 12$, every Fibonacci number $F_n$ is divisible by a prime that is not a divisor of any smaller Fibonacci number.
\begin{theorem}[Williams \cite{Williams}]
    For a prime $p$, $p \mid F_{p-\left(\frac{5}{p}\right)}$ where $\left(\frac{5}{p}\right)$ is the Legendre symbol.
\end{theorem}
\begin{theorem}[Carmichael \cite{Carmichael}]
    For every index $n\neq1,2,6,12$, there exists a prime $p$ such that $p \mid F_n$ and for all $m<n$, $p \not\mid F_m$.
\end{theorem}
\begin{example*}
Table \ref{Table_omega} categorizes the integer bases according to the number of zeros in their Pisano period. Compare this with Table \ref{table2}, the prime factorization of the first few Fibonacci numbers - for all $n\vargeq5$, the prime factors of $F_n$ are sorted based on the congruence class of $n$. 

\begin{table}[H]
\centering
\caption{Splitting of integers according to the number of zeros in one Pisano period}
\begin{tabular}{c|c|c}
$\omega(m)=1$&$\omega(m)=2$&$\omega(m)=4$\\
$n\equiv2\pmod{4}$&$n\equiv0\pmod{4}$&$n\equiv\pm1\pmod{4}$\\
\hline
1, 2, 4, 11, 19, 22, 29, 31,&3, 6, 7, 8, 9, 12, 14, 15,&5, 10, 13, 17, 25, 26, 34, 37,\\
38, 44, 58, 59, 62, 71, 76, 79,&16, 18, 20, 21, 23, 24, 27, 28,&50, 53, 61, 65, 73, 74, 85, 89,\\
101, 116, 118, 121, 124, 131,&30, 32, 33, 35, 36, 39, 40,&97, 106, 109, 113, 122, 125,\\
 139,142, 151, 158, 179, 181,\ldots&41, 42, 43, 45, 46, 47, 48, 49,\ldots&130, 137, 146, 149, 157\ldots\\
&&\\
OEIS: A053031 \cite{[A053031]}&OEIS: A053030 \cite{[A053030]}&OEIS: A053029 \cite{[A053029]}
\end{tabular}
\label{Table_omega}
\end{table}
 
\begin{table}[ht!]
    \centering
    \begin{tabular}{c|c|c|c|c|c|c|c|c|c|c|c|c|c|c|c|c}
         $n$&1&2&3&4&5&6&7&8&9&10&11&12&13&14&15&16  \\
         \hline
         $F_n$&1&1&2&3&5&$2^3$&13&$3\cdot7$&$2\cdot17$&$5\cdot11$&89&$2^4\cdot3^2$&233&$13\cdot29$&$2\cdot5\cdot61$&$3\cdot7\cdot47$ 
    \end{tabular}
    \caption{Prime factorization of $F_n$ for $n\varleq16$.}
    \label{table2}
\end{table}
\end{example*}

For $n\vargeq5$, notice that whenever $n\equiv\pm1\pmod{4}$, the odd prime factors of $F_n$ fall in the $\omega(m)=4$ sequence, whenever $n\equiv0\pmod{4}$, the odd prime factors of $F_n$ fall in the $\omega(m)=2$ sequence (unless the prime factor is also a prime factor of another $F_n$ for an odd index $n$), and whenever $n\equiv2\pmod{4}$, the odd prime factors of $F_n$ fall in the $\omega(m)=1$ sequence (unless the prime factor is also a prime factor of another $F_n$ for an odd index $n$).

\section{Preliminaries}
The proof of Theorem \ref{main_theorem} follows from many established results. This section details the tools used in the proofs that follow. There is a relationship between the rank, the order, and the Pisano period that was first established by Wall \cite{Wall}:
\begin{theorem}[Wall \cite{Wall}]\label{pi = alpha x omega}
    $\pi(p)=\alpha(p)\omega(p)$.
\end{theorem}
\noindent A table of relations was compiled by Vinson \cite{Vinson} and later reformulated into a theorem of Renault \cite{Renault} categorizing when a number has exactly $1$, $2$, or $4$ zeros in its Pisano period. 
\begin{theorem}[Renault \cite{Renault}]\label{thm:Renault}
Let $m$ and $n$ be positive integers, then $\omega\left(\text{lcm}[m,n]\right)$ is given by Table \ref{VinsonTable}:

\begin{table}[H]
\centering
\begin{tabular}{cc|ccc}
&&&$\omega(m)$&\\
&&&&\\
&&1&2&4\\
\hline
&&&&\\
&1&1&2&4 if $m=2$, else 2\\
&&&&\\
$\omega(n)$&2&2&2&2\\
&&&&\\
&4&4 if $n=2$, else 2&2&4
\end{tabular}
\caption{Table of $\omega\left(\text{lcm}[m,n]\right)$.}\label{VinsonTable}
\label{Table of Omega}
\end{table}
\end{theorem}
There are theorems found in the work of Renault \cite{Renault} that relate the order of $m$ with the rank and the period of $m$; Theorem \ref{omega_renault} collects these results in a single result.
\begin{theorem}[Renault \cite{Renault}]\label{omega_renault} \ \\
    \begin{itemize}
        \item For $m>3$, $\omega(m)=4$ if and only if $\alpha(m)\equiv\pm1\pmod{4}$.\\
        \item $\omega(m)=2$ if and only if $4\mid\pi(m)$ and $2\mid\alpha(m)$.\\
        \item $\omega(m)=1$ if and only if $4\not\mid\pi(m)$.\\
    \end{itemize}
\end{theorem}
A result of Vinson \cite{Vinson} determines the order of powers of 2 in the Fibonacci sequence.
\begin{theorem}[Vinson\cite{Vinson}]\label{omega(2)^x}
    $\omega(2)=\omega(4)=1$, and for $x\vargeq3$, $\omega(2^x)=2$.
\end{theorem}
Renault \cite{Renault} extended Vinson's results, determining the order of any odd prime power.
\begin{theorem}[Renault \cite{Renault}]
    For any odd prime $p$, $\omega(p^e)=\omega(p)$.
\end{theorem}
Daykin \& Dresel\cite{Daykin} discovered the connection between the rank of a prime that divides a Fibonacci index, and the divisibility of that Fibonacci number by the prime.
\begin{theorem}[Daykin \& Dresel\cite{Daykin}]\label{p_and_alpha}
    $p\mid F_n$ if and only if $\alpha(p)\mid n$
\end{theorem}
Another useful result for the proof of Theorem \ref{main_theorem} was found by Wyler, who showed the relationship between the rank of a prime and its Pisano period, according to its congruence class modulo~4.
\begin{theorem}[Wyler \cite{Wyler}]\label{order}\ \\
    \begin{itemize}
        \item If $\alpha(p)\equiv2\pmod{4}$, then $\pi(p)=1\cdot\alpha(p)$.\\
        \item If $\alpha(p)\equiv0\pmod{4}$, then $\pi(p)=2\cdot\alpha(p)$.\\
        \item If $\alpha(p)\equiv\pm1\pmod{4}$, then $\pi(p)=4\cdot\alpha(p)$.
    \end{itemize}
\end{theorem}
Note that the converse of Theorem \ref{order} is also true, allowing Wyler's result to be stated as an if and only if statement.

\begin{corollary}\label{alpha_and_omega}\ \\
    \begin{itemize}
        \item $\omega(p)=4$ if and only if $\alpha(p)\equiv\pm1\pmod{4}$.\\
        \item $\omega(p)=2$ if and only if $\alpha(p)\equiv0\pmod{4}$.\\
        \item $\omega(p)=1$ if and only if $\alpha(p)\equiv2\pmod{4}$.
    \end{itemize}
\end{corollary}
\begin{proof}
    This follows immediately from Theorems \ref{pi = alpha x omega} and \ref{order}.
\end{proof}
While the second and third cases of Corollary \ref{alpha_and_omega} may seem different from Theorem \ref{omega_renault}, a closer inspection reveals that Corollary \ref{alpha_and_omega} implies \ref{omega_renault}. For the first case, when $\alpha(p) \equiv \pm 1 \pmod 4$, the two theorems are identical. If $\alpha(p)\equiv 2 \pmod 4$, then $4 \not \mid \pi(m)$, since $\pi(p)=\alpha(p)$. Likewise, if $\alpha(p) \equiv 0 \pmod 4$, then $2 \mid \alpha(p)$ and $4 \mid \pi(p)$, since $\alpha(p) \mid \pi(p)$.

\subsection{Proof of Conjectures}
With these tools in hand, proof of the conjectures found in the OEIS follow readily. As stated, these conjectures are circular (Conjecture \ref{conj:oeis1} says that 19 has order 1 if and only if 19 has order 1, which is a tautology). Additionally, since 1 is a factor of every $m$, but $\omega(1)=1$, Conjecture \ref{conj:oeis4} would suggest that $\omega(m)$ never equals 4. Hence the idea of factors for these conjectures should exclude $1$ and $m$. The case of primes cannot properly be included in these conjectures. Theorems \ref{new_conj_1} through \ref{new_conj_3} are reformulations of Conjectures \ref{conj:oeis4} and \ref{conj:oeis1}, respectively. 
\begin{theorem}\label{new_conj_1}
    If $m$ is an odd number, all of whose factors (other than $1$ and $m$) have order 4, or $m$ is twice such a number, then $m$ has order $4$.
\end{theorem}
\begin{proof}
    This follows immediately from Theorems \ref{thm:Renault} and \ref{omega(2)^x}.
\end{proof}
\begin{theorem}\label{new_conj_2}
    If $m$ has order $4$, then $m$ is an odd number, all of whose factors (other than 1) have order $4$, or $m$ is twice such a number.
\end{theorem}
\begin{proof}
    Suppose $m$ has order $4$. If $m$ is prime, then all of its factors have order $4$. If $m=2^j\cdot p_1^{e_1}\cdots p_t^{e_t}$, then by Theorem \ref{thm:Renault}, $j\varleq 1$ and $\omega(p_i^{e_i})=4$ for all odd prime factors of $m$. 
\end{proof}

\begin{theorem}\label{new_conj_3}
    If $m$ has order $1$, then $m$ is an odd number, all of whose factors have order 1, or $m$ is twice such a number or four times such a number.
\end{theorem}

\begin{proof}
    Suppose $m=2^j\cdot p_1^{e_1}\cdots p_t^{e_t}$ and $\omega(m)=1$. By Theorem \ref{omega(2)^x}, $j\varleq2$. By Theorem \ref{thm:Renault}, $\omega(ab)=1$ if and only if $\omega(a)=\omega(b)=1$. Hence, $\omega(p_i^{e_i})=1$ for all $1\varleq i \varleq t$. 
\end{proof}

\begin{theorem}
    If $m$ is an odd number, all of whose factors (other than $1$ and $m$) have order 1, or $m$ is twice such a number or four times such a number, then $\omega(m)=1$.
\end{theorem}
\begin{proof}
    This follows immediately from Theorems \ref{thm:Renault} and \ref{omega(2)^x}.
\end{proof}

\section{Proof of Theorem \ref{main_theorem}}
\begin{proof}[Proof. (i)]
Suppose $p \mid F_n$ for $n\equiv\pm1\pmod{4}$. By Lemma \ref{p_and_alpha} this is true if and only if $\alpha(p) \mid n$. And for $n\equiv\pm1\pmod{4}$, this holds if and only if $\alpha(p)\equiv\pm1\pmod4$, which by Corollary \ref{alpha_and_omega} is true if and only if $\omega(p)=4$. From Table \ref{thm:Renault}, $\omega(m)=4$ if and only if $m=2^j\cdot p_1^{e_1}\cdots p_r^{e_r}$ for $j\varleq1$ and where $\omega(p_i^{e_i})=4$ for $1\varleq i \varleq r$.
\end{proof}

\begin{proof}[Proof. (ii)]
Suppose that $p \mid F_n$ for $n\equiv2\pmod4$ and $p$ is not also a divisor of some $F_k$ for an odd index $k$. By Lemma \ref{p_and_alpha}, this is true if and only if $\alpha(p) \mid n$ and $\alpha(p) \not\mid k$ for any odd $k$. Hence, $n\equiv2\pmod4$ if and only if $2\mid\alpha(p)$. Because $\alpha(p)\mid n$ and $4\not\mid n$, it follows that $4\not\mid \alpha(p)$, so $\alpha(p)\equiv2\pmod4$. By Corollary \ref{alpha_and_omega}, $\alpha(p)\equiv2\pmod4$ if and only if $\omega(p)=1$. From Table \ref{thm:Renault}, $\omega(m)=1$ for $m=2^j\cdot p_1^{e_1}\cdots p_r^{e_r}$ whenever $j\vargeq3$ and every $p_i^{e_i}$ has order 1.
\end{proof}

\begin{proof}[Proof. (iii)]
This follows directly from the proof of (i) and (ii), but to illuminate the situation, consider $m=2^jp_1^{e_1}\cdots p_t^{e_t}$. By Theorems \ref{omega(2)^x} and \ref{thm:Renault}, if $j\vargeq3$, then $\omega(m)=2$, and if even a single $p_i^{e_i}$ has order $2$, then $\omega(m)=2$. This absorbs all divisors of $F_n$ for indices $n\equiv0\pmod4$ as well as multiples of $2^x$ times any divisor of $F_n$ for $x\vargeq3$, among other numbers.
\end{proof}

\section{K-Fibonacci and (a,b)-Fibonacci Sequences}

Generalizations of the Fibonacci sequence are categorized according to which variables of a binary recurrence sequence are fixed. In particular, Renault \cite{Renault2} extends many properties of the Pisano period to $(a,b)$-Fibonacci sequences. These are sequences where $F_0=0$, $F_1=1$, and $F_n=aF_{n-1}+bF_{n-2}$ for positive integers $a$ and $b$. Unlike the Fibonacci sequence where $\omega(m)$ can only equal 1, 2, or 4, such loose restrictions for $a$ and $b$ allow the number of zeros in a Pisano period to take on infinitely many values.
\begin{theorem}[Renault \cite{Renault2}]\label{(a,b)-order}
    $\omega(m)\mid2\cdot\text{ord}_m(-b)$
\end{theorem}
The reason $\omega(m)$ only takes values 1, 2, and 4 in the Fibonacci sequence is because $b=1$ and $\text{ord}_m(-1)=2$ for all $m \vargeq 3$. Hence, $\omega(m) \mid 4$ and 1, 2, and 4 are the only divisors. There is a well-studied category of binary recurrences where $b$ is fixed at $1$ while $a$ varies. Traditionally in this case, $a$ is swapped out for $K$, and these binary recurrences are known as $K$-Fibonacci sequences, where $F_{K,0}=0$, $F_{K,1}=1$, and $F_{K,n} = KF_{K,n-1}+F_{K,n-2}$. Because $b$ is fixed at 1, $\omega(m)=1$, 2, or 4 for all $K$. Denote the $K$-Pisano period, $K$-order, and $K$-rank of a positive integer $m$ by $\pi_K(m)$, $\omega_K(m)$, and $\alpha_K(m)$, respectively. Several important properties generalize to all $(a,b)$-Fibonacci sequences, which are stated here as they pertain to $K$-Fibonacci sequences:

\begin{theorem}\label{period-product}(Renault \cite{Renault2})
    For all $K$ and $m$, $\pi_K(m)=\alpha_K(m)\omega_K(m)$.
\end{theorem}

\begin{theorem}\label{period-rank-lcm}(Renault \cite{Renault2})
    For all $K$ and $m$, 
    \begin{itemize}
        \item $\pi_K(\lcm[m,n])=\lcm[\pi_K(m),\pi_K(n)]$\\
        \item $\alpha_K(\lcm[m,n])=\lcm[\alpha_K(m),\alpha_K(n)]$.
    \end{itemize}
\end{theorem}
There is a class of sequences that plays a role similar to that of the Lucas sequence for any generalized Fibonacci sequence. Fiebig, Mbirika, \& Spilker \cite{FMS} define a companion sequence for any $(a,b)$-Fibonacci sequence; in particular, the $K$-Lucas sequence is defined as the companion sequence for the $K$-Fibonacci sequence:

\begin{definition}\label{K-lucas}
    The $K$-Lucas sequence is the binary recurrence sequence $L_{K,n}$ with $L_{K,0}=2$, $L_{K,1}=K$, and $L_{K,n}=KL_{K,n-1}+L_{K,n-2}$ for $n \vargeq 2$.
\end{definition}

The $K$-Lucas sequence has many remarkable relationships to the $K$-Fibonacci sequence. Presented here are two identities that will be useful later.

\begin{lemma}\label{lucas}(Fiebig, Mbirika, and Spilker \cite{FMS})
    For any $K$ and $n$, the following holds:
    \begin{itemize}
        \item $F_{K,2n} = F_{K,n}L_{K,n}$
        \item $L_{K,n} = F_{K,n+1}+F_{K,n-1}$.
    \end{itemize}
\end{lemma}
To extend these results to $K$-Fibonacci sequences, here some new notation is introduced. 
\begin{definition}
    Let the first number that appears after the first zero in the Pisano period modulo $m$ (excluding the initial zero) be the $K$-Fibonacci residue of $m$, denoted $\beta_K(m)$.
\end{definition}
For example, the residue of $5$ in the classic Fibonacci sequence is $3$ (in other notation, $\beta_1(5)=3$) because the sequence modulo 5 begins $0, 1, 1, 2, 3, 0, \boxed{3}, \ldots$. The residue has an interesting connection to the order, which follows from a preliminary lemma:
\begin{lemma}\label{residue-subtraction}
    For any $m$, $n$, and $K$, where $n \vargeq \alpha_K(m)$, $F_{K,n}\equiv\beta_K(m)F_{K,n-\alpha_K(m)} \pmod m$.
\end{lemma}
\begin{proof}
    It is useful to remark that $\alpha_K(m)$, or the rank, is the index of the first zero in the Pisano period modulo $m$. Note that $\beta_K(m) \equiv F_{K,\alpha_K(m)+1} \pmod m$. As a base case for an inductive proof, note that $F_{K,\alpha_K(m)} \equiv 0 \equiv \beta_K(m)F_{K,0} \pmod m$ and $F_{K,\alpha_K(m)+1}\equiv\beta_K(m)\equiv\beta_K(m)F_{K,1} \pmod m$, since $F_{K,1}=1$. Then, suppose $F_{K,n}\equiv\beta_K(m)F_{K,n-\alpha_K(m)} \pmod m $ and $F_{K,n-1}\equiv\beta_K(m)F_{K,n-1-\alpha_K(m)} \pmod m$. By definition, \[F_{K,n+1}\equiv K\beta_K(m)F_{K,n-\alpha_K(m)}+\beta_K(m)F_{K,n-1-\alpha_K(m)} \pmod m\]
    \[ = \beta_K(m)(KF_{K,n-\alpha_K(m)}+F_{K,n-1-\alpha_K(m)})\]
    \[ = \beta_K(m)F_{K,n+1-\alpha_K(m)},\]
    which is the statement of the lemma with $n+1$ substituted for $n$.
\end{proof}
Lemma \ref{residue-subtraction} is a $K$-Fibonacci analog of Renault's Identity 3.26 \cite{Renault}.
\begin{proposition}\label{residue-order}
    For any $m$ and $K$, $\omega_K(m)=\text{ord}_m(\beta_K(m))$.
\end{proposition}
\begin{proof}
    The Pisano period is complete when two consecutive entries in the $K$-Fibonacci sequence modulo $m$ are $0$ and $1$. By definition, the number appearing after the first zero, $F_{K,\alpha_K(m)+1}$ is equivalent to $\beta_K(m)$ modulo $m$. By repeated application of Lemma \ref{residue-subtraction}, $F_{K,\alpha_K(m)r+1} = \left(\beta_K(m)\right)^r$. The Pisano period resets when $\left(\beta_K(m)\right)^r \equiv 1 \pmod m$, so $r = \omega_K(m) = \text{ord}_m(\beta_K(m))$.
\end{proof}

The existence of the order establishes that $\gcd\left[\beta_K(m),m\right]=1$. There is also a matrix that parameterizes the $K$-Fibonacci sequence for any $K$, which allows a useful identity for the Fibonacci sequence to be proven for the $K$-Fibonacci sequence:

\begin{theorem}\label{k-fibonacci-matrix} (Cerda-Morales \cite{cerda-morales}).
    Let $U(K)=\left(\begin{matrix}
        K&1\\
        1&0
    \end{matrix}\right)$, then for any integer $n\vargeq1$, \[U^n(K) =
    \left(\begin{matrix}
        F_{K,n+1}&F_{K,n}\\
        F_{K,n}&F_{K,n-1}
    \end{matrix}\right).\] 
\end{theorem}
The classic Fibonacci sequence is said to be \textit{log-Fibonacci}, meaning it is log-convex at even indices and log-concave at odd indices. This property is likewise true for $K$-Fibonacci sequences.
\begin{theorem}\label{log-fibonacci}
    For all $n$ and $K$, $F_{K,n}^2-F_{K,n+1}F_{K,n-1}=(-1)^{n+1}$.
\end{theorem}
\begin{proof}
    Note that det $U^n(K) = (\text{det } U(K))^n = (-1)^n$. But evaluating the determinant explicitly by Theorem \ref{k-fibonacci-matrix} yields $F_{K,n+1}F_{K,n-1}-F_{K,n}^2$. Negating both expressions gives $F_{K,n}^2-F_{K,n+1}F_{K,n-1}=(-1)^{n+1}$.
\end{proof}

Renault uses a more generalized matrix to study $(a,b)$-Fibonacci sequences, but setting $b=1$ provides the same result for $K$-Fibonacci sequences:
\begin{theorem}\label{period-even}(Renault \cite{Renault2}).
    Let $K$ be an integer and let $m > 2$. Then $2 \mid \pi_K(m)$.
\end{theorem}
As an immediate consequence, if the rank is odd, then the order must be even.
\begin{corollary}\label{alpha-odd}
    If $\alpha_K(m)$ is odd, then $2 \mid \omega_K(m)$.
\end{corollary}
\begin{proof}
    By Theorem \ref{period-product}, $\pi_K(m)=\alpha_K(m)\omega_K(m)$. By Theorem \ref{period-even}, $\pi_K(m)$ is even. So at least one of $\alpha_K(m)$ or $\omega_K(m)$ must be even. Hence, when the rank is odd, the order must be even.
\end{proof}

To establish the connection between the index of $K$-Fibonacci numbers and the order, analogs of several other results in the Fibonacci sequence must be developed.

\begin{theorem}\label{negative-fibonacci}
    For a $K$-Fibonacci sequence, $F_{K,\pi_K(m)-n} \equiv (-1)^{n+1}F_{K,n} \pmod m$.
\end{theorem}
\begin{proof}
    This can be shown inductively. As a base case, consider that $F_{K,\pi_K(m)}\equiv 0 \equiv F_{K,0} \pmod m$ and $F_{K,\pi_K(m)-1} \equiv 1 \equiv F_{K,1} \pmod m$. Inductively, suppose $F_{K,\pi_K(m)-(n-1)} = KF_{K,\pi_K(m)-n}+F_{K,\pi_K(m)-(n+1)}$. Rearranging the equation gives $F_{K,\pi_K(m)-(n+1)} = F_{K,\pi_K(m)-(n-1)} = KF_{K,\pi_K(m)-n}$. By the inductive hypothesis, this simplifies to 
    \begin{align*}
    F_{K,\pi_K(m)-(n+1)} &\equiv (-1)^n F_{K,n-1} - (-1)^{n+1}K F_{K,n} \pmod m\\
    &\equiv (-1)^n F_{K,n-1} + K(-1)^n F_{K,n} \pmod m\\
    &\equiv (-1)^n F_{K,n+1} \pmod m\\
    &\equiv (-1)^{n+2} F_{K,n+1} \pmod m.
    \end{align*}
\end{proof}

\begin{theorem}\label{period-divisibility}
    If $F_{K,n} \equiv 0 \pmod m$, then $\alpha_K(m) \mid n$. If $F_{K,n} \equiv 0 \pmod m$ and if $F_{K,n+1} \equiv 1 \pmod m$, then $\pi_K(m) \mid n$.
\end{theorem}
\begin{proof}
    Let $F_{K,n} \equiv 0 \pmod m$ and suppose for contradiction that $\alpha_K(m) \not \mid n$. Then $n = q\alpha_K(m)+r$. By repeated application of Lemma \ref{residue-subtraction}, $0 \equiv F_{K,n} \equiv \beta_K(m)^q F_{K,r} \pmod m$. But $\beta_K(m)$ has an order modulo $m$, hence it is necessary that $\gcd(\beta_K(m),m) = 1$. Therefore $F_{K,r} \equiv 0 \pmod m$ and $\alpha_K(m) \varleq r$, giving a contradiction.\\

    Let $F_{K,n} \equiv 0 \pmod m$ and let $F_{K,n+1} \equiv 1 \pmod m$ and suppose for contradiction that $\pi_K(m) \not \mid n$. Then $n = q\pi_K(m)+r$, where $0 < r < n$. Because $F_{K,r} \equiv F_{K,n} \equiv 0 \pmod m$, and because $F_{K,r+1} \equiv F_{K,n+1} \equiv 1 \pmod m$, it follows that $\pi_K(m) \varleq r$, which again gives a contradiction.\\
\end{proof}
A flurry of interest surrounded the discovery that Pisano periods may lead to a proof of Fermat's Last Theorem. Studied by Wall \cite{Wall}, Sun \& Sun \cite{Sun}, and many others, Wall-Sun-Sun primes for the classic Fibonacci sequence are those primes $p$ where $\pi(p)=\pi(p^2)$. Although none have been found, infinitely many are conjectured to exist. A proof of their non-existence would also prove Fermat's Last Theorem. In the case of $K$-Fibonacci sequences, there is much known about $K$-Wall-Sun-Sun primes, where $\pi_K(p)=\pi_K(p^2)$ (see \cite{Bouazzaoui, Bouazzaoui_2, Harrington, Jones2023, Jones2024}.  
\begin{theorem}[Renault \cite{Renault2}]\label{lifting-exponent-period}
    For all $K$ and all $e \vargeq 1$, $\pi_K(p^{e+1})=p\pi_K(p^e)$ or $\pi_K(p^e)$; in the latter case $p$ is a $K$-Wall-Sun-Sun prime.
\end{theorem}
\begin{corollary}\label{period-power}
    For all $e \vargeq 1$, $\pi_K(p^e) = p^b \pi_K(p)$ for some $0 \varleq b \varleq e-1$.
\end{corollary}
\begin{proof}
    Since the claim is trivial for $e=1$, it is possible to proceed by an inductive argument. By Theorem \ref{lifting-exponent-rank}, $\pi_K(p^{e+1})=\pi_K(p^e)=p^b\pi_K(p) \varleq p^e\pi_K(p)$ or $p\pi_K(p^e)=p^{b+1}\pi_K(p) \varleq p^e\pi_K(p)$. Thus the result holds for $e+1$, completing the proof.
\end{proof}

\begin{theorem}[Renault \cite{Renault2}]\label{lifting-exponent-rank}
    For all $K$ and $e \vargeq 1$, $\alpha_K(p^{e+1})=p\alpha_K(p^e)$ or $\alpha_K(p^e)$.
\end{theorem}
\begin{corollary}\label{rank-power}
    For all $e \vargeq 1$, $\alpha_K(p^e) = p^b \alpha_K(p)$ for some $0 \varleq b \varleq e-1$.
\end{corollary}
\begin{proof}
    This can be seen by induction; a base case of $e=1$ is trivial. By Theorem \ref{lifting-exponent-rank}, $\alpha_K(p^{e+1})=\alpha_K(p^e)=p^b\alpha_K(p) \varleq p^e\alpha_K(p)$ or $p\alpha_K(p^e)=p^{b+1}\alpha_K(p) \varleq p^e\alpha_K(p)$. This shows the result holds for $e+1$, completing the proof.
\end{proof}

\begin{theorem}\label{period-lcm}
    For any $m, n,$ and $K$, $\alpha_K(\lcm[m,n])=\lcm[\alpha_K(m),\alpha_K(n)]$ and $\pi_K(\lcm[m,n])=\lcm[\pi_K(m),\pi_K(n)]$.
\end{theorem}
Renault [\!\!\cite{Renault2}, Theorem 1] proves these identities for the more general $(a,b)$-Fibonacci sequence, where $K$-Fibonacci sequences result from setting $b=1$.

\begin{corollary}[Renault, \cite{Renault2}]\label{divisibility}
    If $m \mid n$, then $\alpha_K(m) \mid \alpha_K(n)$ and $\pi_K(m) \mid \pi_K(n)$.
\end{corollary}

\begin{theorem}\label{omega-prime-powers}
    For any integer $K$, any positive integer $e$ and odd prime $p$, $\omega_K(p^e) = \omega_K(p)$.
\end{theorem}
\begin{proof}
    Renault [\!\!\cite{Renault}, Theorem 3.32] gives a short proof of the fact that $\omega(p^e)=\omega(p)$ for odd primes $p$ in the $1-$Fibonacci sequence, which can now be applied to $K$-Fibonacci sequences. Note that $\omega_K(n)=\frac{\pi_K(n)}{\alpha_K(n)}$ for all $n$. Applying Corollaries \ref{period-power} and \ref{rank-power} gives
    \[\omega_K(p^e) = \frac{\pi_K(p^e)}{\alpha_K(p^e)} = \frac{p^{b_1}\pi_K(p)}{p^{b_2}\alpha_K(p)} = p^{b_1-b_2}\omega_K(p).\] Because the $K$-Fibonacci order can only take on the values $1, 2,$ or $4$, no factors of $p$ may be present in either $\omega_K(p)$ or $\omega_K(p^e)$. Thus $b_1=b_2$ and $\omega_K(p^e)=\omega_K(p)$.
\end{proof}

The following theorem is a direct generalization of the technique and result of Theorem \ref{omega_renault}:

\begin{theorem}\label{omega-alpha}
    Let $K$ be an integer and let $m$ be odd. Then
    \begin{itemize}
        \item For $m>3$, $\omega_K(m)=4$ if and only if $\alpha_K(m)\equiv\pm1\pmod{4}$.\\
        \item $\omega_K(m)=2$ if and only if $4\mid\pi_K(m)$ and $2\mid\alpha_K(m)$.\\
        \item $\omega_K(m)=1$ if and only if $4\not\mid\pi_K(m)$.\\
    \end{itemize}
\end{theorem}
\begin{proof}
    From Lemma \ref{log-fibonacci}, $F^2_{\alpha_K(m)}-F_{\alpha_K(m)+1}F_{\alpha_K(m)-1}=(-1)^{\alpha_K(m)+1}$. This expression can be simplified since $F_{\alpha_K(m)} \equiv 0$ and $F_{\alpha_K(m)+1} \equiv \beta_K(m) \pmod m$ by definition. Furthermore, $F_{\alpha_K(m)+1}=KF_{\alpha_K(m)}+F_{\alpha_K(m)-1} \equiv F_{\alpha_K(m)-1} \pmod m$. Thus, $-\beta_K(m)^2 \equiv (-1)^{\alpha_K(m)+1}$, and
    \begin{align}
    \beta_K(m)^2 \equiv (-1)^{\alpha_K(m)}.
    \end{align}
    Suppose $\alpha_K(m) \equiv \pm 1 \pmod 4$, so $\beta_K(m)^2 \equiv -1 \pmod m$ by equation (1). It follows from Theorem \ref{residue-order} that $2 \not \mid \omega_K(m)$. The only remaining option for $\omega_K(m)$ is 4. For the converse, suppose $\omega_K(m)=4$. Then $\beta_K(m)^2 \equiv -1 \pmod m$, since $\text{ord}_m(\beta_K(m))=\omega_K(m)=4$ by Theorem \ref{residue-order}. Because $\beta_K(m)^2 \equiv (-1)^{\alpha_K(m)} \pmod m$, it follows that $\alpha_K(m)$ is odd. This proves the first part of Theorem \ref{omega-alpha}.
    
    Now suppose that $4 \mid \pi_K(p)$ for an odd prime $p$; then it will be shown that $\omega_K(p) \neq 1$. Let $n = \frac{\pi_K(p)}{2}+1$, which is odd since $4 \mid \pi_K(p)$. Then, $F_{K,n} \equiv (-1)^{n+1} F_{K,\pi_K(p)-n} \equiv F_{K,\pi_K(p)-n} \pmod p$. Because $\pi_K(p)-n = n-2$, the formula $F_{K,n}=KF_{K,n-1}+F_{K,n-2}$ gives $F_{K,n}\equiv KF_{K,n-1}+F_{K,n-2} \pmod p$, so $KF_{K,n-1} \equiv 0 \pmod p$. If $K \equiv 0 \pmod p$, then the sequence will reduce to $0, 1, 0, 1, 0, 1, \ldots$, modulo $p$. Hence, $\pi_K(p) = 2$, contradicting the assumption. It follows that $F_{K,n-1} \equiv 0 \pmod p$. Along with the zero at the beginning of each Pisano period, this shows that $\omega_{K}(p) \vargeq 2$ and hence cannot equal $1$. Therefore, if $\omega_K(p) = 1$, then $4 \not \mid \pi_K(p)$.
    
    Since $\pi_K(p^e)=p^b\pi_K(p)$ by Corollary \ref{period-power} and $p$ is an odd prime, $4 \mid \pi_K(p^e)$ if and only if $4 \mid \pi_K(p)$. If $4 \not \mid \pi_K(p^e)$, then $4 \not \mid \pi_K(p)$, so $\omega_K(p)=\omega_K(p^e)=1$. The proof for a general odd $m$ follows by inducting on the number of distinct prime factors, with the proof for prime powers providing a base case. Suppose $m = p_1^{e_1}p_2^{e_2}\cdots p_r^{e_r}$ and $4 \not \mid \pi_K(m)$. If either $4 \mid \pi_K(p_1^{e_1}p_2^{e_2}\cdots p_{r-1}e^{r-1})$ or $4 \mid \pi_K(p_r^{e_r})$, then by Theorem \ref{period-lcm} $4 \mid m$, contradicting the assumption.
    
    The inductive hypothesis can now be applied to yield $\omega_K(p_1^{e_1}p_2^{e_2}\cdots p_{r-1}e^{r-1})=1$. Therefore, $\pi_K(p_1^{e_1}p_2^{e_2}\cdots p_{r-1}e^{r-1}) = \alpha_K(p_1^{e_1}p_2^{e_2}\cdots p_{r-1}e^{r-1})$. Since $4 \not \mid \pi_K(p_r^{e_r})$, $\omega_K(p_r^{e_r})=1$ and $\pi_K(p_r^{e_r})=\alpha_K(p_r^{e_r})$. Theorem \ref{period-lcm} can be applied once more to give \begin{align*}
    \pi_K(m) = \lcm[\pi_K(p_1^{e_1}p_2^{e_2}\cdots p_{r-1}e^{r-1}),\pi_K(p_r^{e_r})] = \lcm[\alpha_K(p_1^{e_1}p_2^{e_2}\cdots p_{r-1}e^{r-1}),\alpha_K(p_r^{e_r})] = \alpha_K(m),
     \end{align*} which proves that $\omega_K(m)=1$ for any odd $m$.
    The converse of the above statement can be proven by assuming that $\omega_K(m)=2$ or $4$. Since $\alpha_K(m)$ is odd if and only if $\omega_K(m)=1$, then $\alpha_K(m)$ is even; thus $4 \mid \alpha_K(m)\omega_K(m) = \pi_K(m)$. If $4 \not \mid \pi_K(m)$, then $\omega_K(m)=1$, proving the third part of the theorem. The second part follows immediately from the first and the third.
\end{proof}

\begin{theorem}\label{K-wyler}
    Let $K$ be an integer and let $p$ be an odd prime; then,
    \begin{itemize}
        \item $\omega_K(p)=4$ if and only if $\alpha_K(p)\equiv\pm1\pmod{4}$.\\
        \item $\omega_K(p)=2$ if and only if $\alpha_K(p)\equiv0\pmod{4}$.\\
        \item $\omega_K(p)=1$ if and only if $\alpha_K(p)\equiv2\pmod{4}$.
    \end{itemize}
\end{theorem}
\begin{proof}
    The first item is equivalent to its corresponding part in Theorem \ref{omega-alpha}. Now suppose that $\alpha_K(p) \equiv 0 \pmod 4$. Then $2 \mid \alpha_K(p)$ and $4 \mid \pi_K(p)$, since $\alpha_K(p) \mid \pi_K(p)$. Thus, Theorem \ref{omega-alpha} can be applied to show that $\omega_K(p)=2$.
    
    Finally, consider the case where $\alpha_K(p) \equiv 2 \pmod 4$. Recall that by Lemma \ref{negative-fibonacci}, $F_{K,\pi_K(p)-n} \equiv (-1)^{n+1}F_{K,n} \pmod m$. By Lemma \ref{residue-subtraction},
    \[F_{K,\alpha_K(p)-n} \equiv \left(\beta_K(p)\right)^{1-\omega_K(p)}(-1)^{n+1}F_{K,n} \pmod p.\]
    Substituting $n = \frac{\alpha_K(p)}{2}$ gives
    \[F_{K,\frac{\alpha_K(p)}{2}} \equiv \left(\beta_K(p)\right)^{1-\omega_K(p)}(-1)^{\frac{\alpha_K(p)}{2}+1}F_{K,\frac{\alpha_K(p)}{2}} \pmod p\]
    and 
    \[1 \equiv \left(\beta_K(p)\right)^{1-\omega_K(p)}(-1)^{\frac{\alpha_K(p)}{2}+1} \pmod p.\]
    But $\alpha_K(p) \equiv 2 \pmod 4$, so $\frac{\alpha_K(p)}{2}+1$ is even, giving
    \[1 \equiv \left(\beta_K(p)\right)^{1-\omega_K(p)} \pmod p\]
    and
    \begin{align}\label{ord_beta}
        \text{ord}_p(\beta_K(p)) = \omega_K(p) \mid 1-\omega_K(p).
    \end{align}
    The only possibility for $\omega_K(p)$ that satisfies equation \eqref{ord_beta} is 1, completing the proof.
\end{proof}

\begin{theorem}[Renault \cite{Renault2}]\label{Renault-powersof2}
    $\pi_K(2^e)=2^{e-e'}\pi_K(4)$, where $2\varleq e' \varleq e$ is maximal such that $\pi_K(2^{e'}) = \pi_K(4)$, and
    $\alpha_K(2^e)=2^{e-e'}\alpha_K(4)$, where $2\varleq e' \varleq e$ is maximal such that $\alpha_K(2^{e'}) = \alpha_K(4)$.
\end{theorem}

Some other properties of the order are dependent upon the parity of $K$, where sequences with odd $K$ have the most similarities to the Fibonacci sequence, because it is the case $K=1$.

\subsection{K-Fibonacci Sequences for Odd K}

Any $K$-Fibonacci sequence with odd $K$ is equivalent to the classical Fibonacci sequence modulo 2, since the formula $F_{K,n}=KF_{K,n-1}+F_{K,n-2}$ can be reduced modulo 2. Properties involving powers of 2 are frequently the same as in the Fibonacci sequence, as seen by the behavior of the period, rank, and order.

\begin{theorem}\label{omegaK(2)^x-odd}
    If $K$ is odd, $\omega_K(2)=\omega_K(4)=1$, and for $x\vargeq3$, $\omega_K(2^x)=2$, precisely as in Theorem \ref{omega(2)^x} for the classical Fibonacci sequence.
\end{theorem}
\begin{proof}
Note that the $K$-Fibonacci sequence modulo $m$ is equivalent to the $K+am$-Fibonacci sequence for any integer $a$, with the same period, rank, and order. Hence, the pattern of $K$-Fibonacci sequences modulo $m$ can be numerically determined by computing for $K=0, 1, \ldots, m-1$. This confirms Table \ref{table4} for odd $K$:
    \begin{table}[ht!]
    \centering
    \begin{tabular}{c|c|c|c}
         $n$&$\pi_K(n)$&$\alpha_K(n)$&$\omega_K(n)$  \\
         \hline
         $2^1$&$3$&$3$&$1$\\
         \hline
         $2^2$&$6$&$6$&$1$\\
         \hline
         $2^3$&$12$&$6$&$2$\\
         \hline
         $2^4$&$24$&$12$&$2$
    \end{tabular}
    \caption{Period, rank, and order for odd $K$ modulo powers of 2}
    \label{table4}
\end{table}

In Theorem \ref{Renault-powersof2} it is clear from the table that $e'=2$ with respect to the period and $e'=3$ with respect to the rank. Hence $\pi_K(2^x)=2^{x-2}\pi_K(4) = 2^{x-3}\pi_K(8)$ and $\alpha_K(2^x)=2^{x-3}\alpha_K(8)$ for $x \vargeq 3$. Therefore, $\omega_K(2^x)=\frac{\pi_K(2^x)}{\alpha_K(2^x)} = \frac{2^{x-3}\pi_K(8)}{2^{x-3}\alpha_K(8)} = \omega_K(8) = 2$.
\end{proof}
\ \\ %%Extra line to keep Table 4 above Theorem 4.26, otherwise it bleeds between the statement of the Theorem.%%

\begin{theorem}\label{omega-odd-K}
    Let $K$ be odd and let $m > 1$. Then Theorem \ref{omega-alpha} holds:
    \begin{itemize}
        \item For $m>3$, $\omega_K(m)=4$ if and only if $\alpha_K(m)\equiv\pm1\pmod{4}$.\\
        \item $\omega_K(m)=2$ if and only if $4\mid\pi_K(m)$ and $2\mid\alpha_K(m)$.\\
        \item $\omega_K(m)=1$ if and only if $4\not\mid\pi_K(m)$.\\
    \end{itemize}
\end{theorem}
\begin{proof}
    Let $m = 2^\gamma n$, where $\gamma > 0$ and $n$ is odd, since the $\gamma = 0$ case was already proven in Theorem \ref{omega-alpha}. If $\gamma=1$, then $m = 2n$. Then, $\pi_K(m) = \lcm[3,\pi_K(n)]$ and $\alpha_K(m) = \lcm[3,\alpha_K(n)]$. Suppose $\omega_K(n)=4$. Then $\alpha_K(n) \equiv \pm 1 \mod 4$, and $\pi_K(n)=4\alpha_K(n)$. Substituting this into the least common multiple formulas gives $\pi_K(m) = \lcm[3,4\alpha_K(n)] = 4\cdot \lcm[3,\alpha_K(n)] = 4\alpha_K(m)$, so $\omega_K(m)=4$, as desired. If $4 \mid \pi_K(n)$ and $2 \mid \alpha_K(n)$, then $\omega_K(n)=2$ by Theorem \ref{omega-alpha}. Hence, $\pi_K(m)=\lcm[3,2\alpha_K(n)] = 2\cdot\lcm[3,\alpha_K(n)]=2\alpha_K(m)$. Likewise, $4 \mid \pi_K(m)$ and $2 \mid \alpha_K(m)$ by Corollary \ref{divisibility}. Finally, if $4 \not \mid \pi_K(n)$ and $\omega_K(n)=1$, then $\pi_K(n)=\alpha_K(n)$ and $\pi_K(m)=\alpha_K(m)=\lcm[3,\alpha_K(n)]$. Since no factor of $4$ is introduced into the least common multiple, and $\omega_K(m)=1$, the theorem is satisfied.

    If $\gamma=2$, then $m = 4n$. Once again, $n$ is odd and Theorem \ref{omega-alpha} applies to it. By Table 5, $\pi_K(4)=\alpha_K(4)=6$ for all odd $K$. If $\omega_K(n)=4$, this gives $\pi_K(m)=\lcm[6,4\alpha_K(n)]$ and $\alpha_K(m)=\lcm[6,\alpha_K(n)]$. By Theorem \ref{omega-alpha}, $\alpha_K(n)$ is odd. Either $\alpha_K(m)=2\alpha_K(n)$, $\pi_K(m)=4\alpha_K(n)$ and $\omega_K(m)=2$ (if $3 \mid n$), or $\alpha_K(m)=6\alpha_K(n)$, $\pi_K(m)=12\alpha_K(n)$, and $\omega_K(m)=2$ (if $3 \not \mid n$). In each case the theorem to be proved holds. If $\omega_K(n)=2$, $4 \mid \pi_K(m)$, and $2 \mid \alpha_K(m)$. There are again two sub-cases, depending on whether or not $3 \mid \alpha_K(n)$. If $3 \mid \alpha_K(n)$, then $\pi_K(m)=\lcm[6,2\alpha_K(n)]=2\alpha_K(n)$, $\alpha_K(m)=\lcm[6,\alpha_K(n)]=\alpha_K(n)$, and $\omega_K(m)=2$. Otherwise, $\pi_K(m)=6\alpha_K(n)$, $\alpha_K(m)=3\alpha_K(n)$, $\omega_K(m)=2$, and once again the theorem is satisfied. Finally, suppose $\omega_K(n)=1$; hence $4 \not \mid \pi_K(n)$. In this case, $\pi_K(m)=\lcm[6,\alpha_K(n)]=\alpha_K(m)$. No factor of 4 is introduced in the least common multiple, so $4 \not \mid \pi_K(m)$ and $\omega_K(m)=1$, as was to be proved.

    Now suppose $\gamma \vargeq 3$. By Theorem \ref{omegaK(2)^x-odd} $\pi_K(2^{\gamma}) = 2\alpha_K(2^{\gamma})$ and $\alpha_K(2^{\gamma})=3\cdot2^{\gamma-2}$. If $\omega_K(n)=4$, then $\alpha_K(n)$ is odd. This gives the equations $\pi_K(m)=\lcm[3\cdot2^{\gamma-1},4\alpha_K(n)]$ and $\alpha_K(m)=\lcm[3\cdot2^{\gamma-2},\alpha_K(n)]$. These differ by a factor of 2, and hence $\omega_K(m)=2$, because $2^{\gamma-1}$ and $2^{\gamma-2}$ overwhelm the powers of 2 in the latter terms. If $\omega_K(n)=2$, then $4 \mid \pi_K(n)$ and $2 \mid \alpha_K(n)$. Hence $\pi_K(m)=\lcm[3\cdot2^{(\gamma-2)+1},2\alpha_K(n)]$ and $\alpha_K(m)=\lcm[3\cdot2^{\gamma-2},\alpha_K(n)]$. Notice that each term in the least common multiple for $\pi_K(m)$ is twice that of the corresponding term for $\alpha_K(m)$; thus $\omega_K(m)=2$. Finally, suppose $\omega_K(n)=1$. Let $\pi_K(n)=\alpha_K(n)=2^a b$, where $b$ is odd. Then $\pi_K(m)=\lcm[3\cdot2^{\gamma-1},2^{a+1}b]$ and $\alpha_K(m)=\lcm[3\cdot2^{\gamma-2},2^a b]$. By Theorem \ref{omega-odd-K}, $a < 2$, so the largest power of 2 in each least common multiple comes from the left side. Other factors can be ignored, since $\omega_K(m)=1$, 2, or 4. Therefore, $\omega_K(m)=\frac{2^{\gamma-1}}{2^{\gamma-2}}=2$. Theorem \ref{omegaK(2)^x-odd} shows that $4 \mid \pi_K(2^{\gamma})$ and $2 \mid \alpha_K(2^{\gamma})$. By Corollary \ref{divisibility}, $4 \mid \pi_K(m)$ and $2 \mid \alpha_K(m)$, and the theorem is satisfied.
\end{proof}

\begin{theorem}\label{K-multiplication}
Let $K$ be odd and $m, n > 1$ be integers; alternatively, let $K$ be even and suppose $m,n > 1$ are odd. Then the following table holds:
\begin{table}[H]
\centering
\begin{tabular}{cc|ccc}
&&&$\omega_K(m)$&\\
&&&&\\
&&1&2&4\\
\hline
&&&&\\
&1&1&2&4 if $m=2$, else 2\\
&&&&\\
$\omega_K(n)$&2&2&2&2\\
&&&&\\
&4&4 if $n=2$, else 2&2&4
\end{tabular}
\caption{Table of $\omega_K\left(\text{lcm}[m,n]\right)$ for odd $K$ and/or odd $m, n$.}
\end{table}
\end{theorem}
\begin{proof}
    Let $\omega_K(m)=\omega_K(n)=1$. Applying Theorems \ref{omega-alpha} and \ref{omega-odd-K} gives $4 \not \mid \alpha_K(m)$ and $4 \not \mid \alpha_K(n)$; therefore $4 \not \mid \alpha_K(\lcm[m,n]) = \lcm\left[\alpha_K(m),\alpha_K(n)\right]$. Using Theorems \ref{omega-alpha} and \ref{omega-odd-K} again yields $\omega_K(\lcm[m,n])=1$.

    Suppose $\omega_K(m)=2$; $\omega_K(n)$ may be $1$, $2$, or $4$. Then $2 \mid \alpha_K(n)$, so $2 \mid \lcm\left[\alpha_K(m),\alpha_K(n)\right]=\alpha_K(\lcm[m,n])$. Since $4 \mid \pi_K(m)$, $4 \mid \lcm\left[\pi_K(m),\pi_K(n)\right]=\pi_K(\lcm[m,n])$. Therefore $\omega_K(\lcm[m,n])=2$.

    Let $\omega_K(m)=1$ and $\omega_K(n)=4$. Since $4 = \pi_K(n)=\omega_K(n)\alpha_K(n)$, it follows that $4 \mid \pi_K(n)$. Thus $4 \mid \pi_K(\lcm[m,n])$. It follows that $\omega_K(\lcm[m,n]) = 2$ or $4$. If $2 \mid \alpha_K(\lcm[m,n])$, then $\omega_K(\lcm[m,n])=2$ by Theorems \ref{omega-alpha} and \ref{omega-odd-K}, otherwise $\omega_K(\lcm[m,n])=4$. Suppose $m = 2$; then $\pi_K(m)=3$ and $\alpha_K(m)=1$. Since $\alpha_K(n) \equiv 1 \pmod 2$, it follows that $2 \not \mid \alpha_K(\lcm[m,n])$ and $\omega_K(\lcm[m,n])=4$. Whereas if $m > 2$, then $2 \mid \pi_K(m)$ by Corollary \ref{period-even}. Since $\omega_K(m)=1$, then $2 \mid \alpha_K(m)$ and $\omega_K(\lcm[m,n])=2$.

    Finally, assume $\omega_K(m)=\omega_K(n)=4$, so $\alpha_K(m) \equiv \alpha_K(n) \equiv 1 \pmod 2$. Therefore, $\alpha_K(\lcm[m,n])=\lcm[\alpha_K(m),\alpha_K(n)]$ is odd as well, and applying Theorems \ref{omega-alpha} and \ref{omega-odd-K} again yields $\omega_K(\lcm[m,n])=4$.
\end{proof}

It is now possible to prove the main theorem not only for the Fibonacci sequence, but for all $K$-Fibonacci sequences where $K$ is odd.

\begin{proof}[Proof. (i)]
Suppose $p \mid F_{K,n}$ for $n\equiv\pm1\pmod{4}$. By Lemmas \ref{residue-subtraction} and \ref{period-divisibility}, this is true if and only if $\alpha_K(p) \mid n$. And for $n\equiv\pm1\pmod{4}$, this holds if and only if $\alpha_K(p)\equiv\pm1\pmod4$, which by Theorem \ref{K-wyler} is true if and only if $\omega_K(p)=4$. From Table \ref{K-multiplication}, $\omega_K(m)=4$ if and only if $m=2^j\cdot p_1^{e_1}\cdots p_r^{e_r}$ for $j\varleq1$ and where $\omega_K(p_i^{e_i})=4$ for $1\varleq i \varleq r$.
\end{proof}

\begin{proof}[Proof. (ii)]
Suppose that $p \mid F_{K,n}$ for $n\equiv2\pmod4$ and $p$ is not also a divisor of some $F_{K,t}$ for an odd index $t$. By Lemmas \ref{residue-subtraction} and \ref{period-divisibility}, this is true if and only if $\alpha_K(p) \mid n$ and $\alpha_K(p) \not\mid t$ for any odd $t$. Hence, $n\equiv2\pmod4$ if and only if $2\mid\alpha_K(p)$. Because $\alpha_K(p)\mid n$ and $4\not\mid n$, it follows that $4\not\mid \alpha_K(p)$, so $\alpha_K(p)\equiv2\pmod4$. By Theorem \ref{K-wyler}, $\alpha_K(p)\equiv2\pmod4$ if and only if $\omega_K(p)=1$. From Table \ref{K-multiplication}, $\omega_K(m)=1$ for $m=2^j\cdot p_1^{e_1}\cdots p_r^{e_r}$ whenever $j\vargeq3$ and every $p_i^{e_i}$ has order 1.
\end{proof}

\begin{proof}[Proof. (iii)]
This follows directly from the proof of (i) and (ii), but to illuminate the situation, consider $m=2^jp_1^{e_1}\cdots p_t^{e_t}$. By Theorems \ref{omegaK(2)^x-odd} and \ref{K-multiplication}, if $j\vargeq3$, then $\omega_K(m)=2$, and if even a single $p_i^{e_i}$ has order $2$, then $\omega_K(m)=2$. This absorbs all divisors of $F_{K,n}$ for indices $n\equiv0\pmod4$ as well as multiples of $2^x$ times any divisor of $F_{K,n}$ for $x\vargeq3$, among other numbers.
\end{proof}

\subsection{K-Fibonacci Sequences for Even K}

The properties of $K$-Fibonacci sequences for odd numbers hold for both odd and even $K$. However, when both $K$ and $n$ are even, different properties occur for $\omega_K(n)$.

\begin{theorem}\label{omegaK(2)^x-even}
    Let $\nu_2(n)$ be the 2-adic valuation. If $K$ is even and $x \vargeq 1$, then $\pi_K(2^x)=\alpha_K(2^x)=2^{x+1-\nu_2(\gcd(2^x,K))}$ and $\omega_K(2^x)=1$.
\end{theorem}
\begin{proof}
    The proof proceeds similarly to Theorem \ref{omegaK(2)^x-odd}; numerical computations yield the first two rows of the table below for all even $K$. The result can then be proven for higher powers by induction.

    Suppose $2^{x+1} \mid K$. Modulo $K$, this sequence is equivalent to the $K=0$ case, which is the rather uninteresting sequence $0, 1, 0, 1, \ldots$, with period 2, rank 2, and order 1. Since $x+2-\nu_2(\gcd(2^{x+1},K))=x+1-\nu_2(\gcd(2^{x},K))$, the theorem is satisfied.

    Now suppose $2^{x+1} \not \mid K$. By Theorems \ref{lifting-exponent-period} and \ref{lifting-exponent-rank}, $\pi_K(2^{x+1})=\pi_K(2^x)$ or $2\pi_K(2^x)$, and likewise for the rank. Assume towards a contradiction that $\alpha_K(2^{x+1})=\alpha_K(2^x)=2\rho$ for brevity. 

    First assume that $\rho = 1$. Then $\alpha_K(2^{x+1})=\alpha_K(2^x)=2$. The recurrence gives $F_{K,2}=K$; hence $2^{x+1} \mid K$, which contradicts the assumption that $2^{x+1} \not \mid K$. Applying the recurrence again gives $F_{K,3}=K^2+1 \equiv 1 \pmod{2^{x+1}}$, so $\pi_K(2^{x+1})=\alpha_K(2^{x+1})=2$ and $\omega_K(2^{x+1})=1$.
    
    If $\rho > 1$, then $F_{K,2\rho}=F_{K,\rho}L_{K,\rho}=F_{K,\rho}(F_{K,\rho+1}+F_{K,\rho-1}) \equiv 0 \pmod{2^{x+1}}$ by Lemma \ref{lucas}. By Lemma \ref{negative-fibonacci}, $F_{K,\rho-1}=F_{K,\rho+1}$; hence $2F_{K,\rho}F_{K,\rho+1} \equiv 0 \pmod{2^{x+1}}$. But $F_{K,\rho+1}$ is odd since $\rho+1$ is odd; therefore $2^x \mid F_{K,\rho}$, contradicting the fact that $\alpha_K(2^x)=2\rho$.

\end{proof}
    
\begin{table}[]
    \centering
    \begin{tabular}{c|c|c|c|c}
         $\pi_K(n)$&$K=2$&$4$&$6$&$8$  \\
         \hline
         $n=2^1$&2&2&2&2\\
         \hline
         $2^2$&4&2&4&2\\
         \hline
         $2^3$&8&4&8&2\\
         \hline
         $2^4$&16&8&16&4
    \end{tabular}
    \caption{Pisano periods for various even $K$}
    \label{table5}
\end{table}

\begin{theorem}\label{omega-even-K}
    Let $K$ be even and let $m > 1$ be such that $4 \not \mid m$. Then Theorem \ref{omega-alpha} holds:
    \begin{itemize}
        \item For $m>3$, $\omega_K(m)=4$ if and only if $\alpha_K(m)\equiv\pm1\pmod{4}$.\\
        \item $\omega_K(m)=2$ if and only if $4\mid\pi_K(m)$ and $2\mid\alpha_K(m)$.\\
        \item $\omega_K(m)=1$ if and only if $4\not\mid\pi_K(m)$.\\
    \end{itemize}
\end{theorem}
\begin{proof}
    It may be assumed that $m$ is even, since the theorem was proven for odd $m$ in Theorem \ref{omega-alpha}. Let $m = 2n$, where $n$ is odd as $4 \not \mid m$. Then $\pi_K(m) = \lcm[1,\pi_K(n)]=\pi_K(n)$ and $\alpha_K(m) = \lcm[1,\alpha_K(n)]=\alpha_K(n)$. Therefore, neither the period, rank, nor order change when an odd number is multiplied by 2, and the theorem is satisfied.
\end{proof}

If $m$ is restricted to odd numbers, the ``multiplication table'' for $K$-Fibonacci sequences is the same for even $K$ as for odd $K$. However, there are some slight differences when $m$ and $n$ are both even:

\begin{theorem}\label{table_even}
Let $m$, $n$, and $K$ be positive integers where $K$ is even, then:
\begin{table}[]
\centering
\begin{tabular}{cc|ccc}
&&&$\omega_K(m)$&\\
&&&&\\
&&1&2&4\\
\hline
&&&&\\
&1&1&1 or 2&1 or 2\\
&&&&\\
$\omega_K(n)$&2&1 or 2&2&2\\
&&&&\\
&4&1 or 2&2&4
\end{tabular}
\caption{Table of $\omega_K\left(\text{lcm}[m,n]\right)$ for even $K$.}
\end{table}
\end{theorem}

\begin{proof}
    The diagonal cases are simple to prove. Suppose $\omega_K(m)=\omega_K(n)=t$; then $\pi_K(m)=t\alpha_K(m)$ and $\pi_K(n)=t\alpha_K(n)$. Hence $\pi_K(\lcm[m,n])=\lcm[\pi_K(m),\pi_K(n)]=\lcm[t\alpha_K(m),t\alpha_K(n)]=t\lcm[\alpha_K(m),\alpha_K(n)]=t\alpha_K(\lcm[m,n])$. Therefore, $\omega_K(\lcm[m,n])=t$.

    Next, suppose that $\omega_K(m)=2$ and $\omega_K(n)=4$. Note that $m$ and $n$ are both odd; by Theorem \ref{omegaK(2)^x-even}, if $m$ were even then $\omega_K(m)=1$. Hence Theorem \ref{K-multiplication} applies and $\omega_K(\lcm[m,n])=2$.
    
    For the other cases, it remains to prove that if $\omega_K(m)=1$, then $\omega_K(\lcm[m,n])\neq4$, as the table is symmetric. Consider the case where $\omega_K(n)=2$. Then $\pi_K(\lcm[m,n])=\lcm[\alpha_K(m),2\alpha_K(n)]$ and $\alpha_K(\lcm[m,n])=\lcm[\alpha_K(m),\alpha_K(n)]$. Hence, either $\pi_K(\lcm[m,n])=2\alpha_K(\lcm[m,n])$ or $\alpha_K(\lcm[m,n])$.
    
    Finally, suppose $\omega_K(n)=4$. This yields $\pi_K(\lcm[m,n])=\lcm[\alpha_K(m),4\alpha_K(n)]$ and $\alpha_K(\lcm[m,n])=\lcm[\alpha_K(m),\alpha_K(n)]$. As $\omega_K(n)\neq1$, Theorem \ref{omegaK(2)^x-even} can once again be applied to show that $n$ is odd. By Theorem \ref{omega-even-K}, $\alpha_K(n)$ is odd; by Corollary \ref{alpha-odd}, $\alpha_K(m)$ is even. Let $\alpha_K(m)=2a$, so that the system above becomes $\pi_K(\lcm[m,n])=\lcm[2a,4\alpha_K(n)]=2\lcm[a,2\alpha_K(n)]$ and $\alpha_K(\lcm[m,n])=\lcm[2a,\alpha_K(n)]$. These differ by at most a factor of 2, so $\omega_K([\lcm[m,n])\neq4$, as was to be proved.
\end{proof}

This allows a weaker version of the main theorem to be extended to $K$-Fibonacci sequences, including ones with even $K$.  The following portions of Theorem \ref{main_theorem} generalize to even $K$. Unfortunately, for even positive integers $m$ it may only be said that $\omega_K(m) \neq 4$, which differs slightly from Theorem $K$.

\begin{theorem}\label{mainthm-even-K}
    For an even integer $K$ and a positive odd integer $m$,
    \begin{itemize}
        \item $\omega_K(m) = 4$ for $m > 3$ if and only if $m$ has prime factorization $m = p_1^{e_1} \cdots p_r^{e_r}$, where for each $1 \varleq i \varleq r$, there is an odd index $n_i$ such that $p_i \mid F_{n_i}$. Additionally, no even numbers exist such that $\omega_K(m)=4$.\\
        \item $\omega_K(m) = 1$ for $m > 3$ if and only if $m$ has prime factorization $m = p_1^{e_1} \cdots p_r^{e_r}$, where all primes $p_i$ are odd and for each $1 \varleq i \varleq r$, there is an index $n_i \equiv 2 \pmod 4$ such that $p_i \mid F_{n_i}$, but $p_i \not \mid F_{\nu_i}$ for all odd $\nu_i$.\\
        \item $\omega_K(m)=2$ for all other odd positive integers $m$.
    \end{itemize}
\end{theorem}

\begin{proof}
It is first useful to note that by Theorem \ref{omegaK(2)^x-even} and Table \ref{table_even}. Let $p$ be an odd prime; since $m$ is odd all prime divisors of $m$ must be odd. Suppose $p \mid F_{K,n}$ for $n\equiv\pm1\pmod{4}$. By Lemmas \ref{residue-subtraction} and \ref{period-divisibility}, which are true for all $K$, $\alpha_K(p) \mid n$. Since $p \mid F_{K,n}$ for some odd $n$, then $\alpha_K(n)$ is odd; thus $\omega_K(p)=4$ by Theorem \ref{K-wyler}, and vice versa.

Now suppose $p \mid F_{K,n}$ for $n\equiv2\pmod4$ but $p$ does not divide $F_{K,t}$ for any odd index $t$. By Lemmas \ref{residue-subtraction} and \ref{period-divisibility}, this is true if and only if $\alpha_K(p) \mid n$ and $\alpha_K(p) \not\mid t$ for any odd $t$. Hence $2\mid\alpha_K(p)$ and $4\not\mid \alpha_K(p)$, so $\alpha_K(p)\equiv2\pmod4$. Theorem \ref{K-wyler} can now be applied to show that $\omega_K(p)=1$, and vice versa. Here it is necessary to assume that $m$ is odd; from Table \ref{K-multiplication} and Theorem \ref{omega-prime-powers}, $\omega_K(m)=1$ if and only if $m=p_1^{e_1}\cdots p_r^{e_r}$ whenever $j\vargeq3$ and every $p_i$ has order 1.

Once again, $\omega_K(m)=2$ for all other odd $m$, as seen by Table \ref{table_even}.
\end{proof}

The limitations of Theorem \ref{mainthm-even-K} (the analog of the main theorem for even $K$) stem from the differences between Theorem \ref{table_even} (the ``multiplication table" for even $K$) and Theorem \ref{K-multiplication} (the ``multiplication table" for odd $K$). For example, in the Pell (2-Fibonacci) sequence, $\omega_2(8)=1$, $\omega_2(10)=2$, but $\omega_2(\lcm[8,10])=1$. Similarly, $\omega_2(5)=4$, but $\omega_2(\lcm[8,5])=1$. There are also pairs of even numbers where the conditions in Theorem \ref{omegaK(2)^x-odd} are satisfied, such as $\omega_2(2)=1$, $\omega_2(12)=2$, and $\omega_2(\lcm[2,12])=2$.
\begin{question}
    For even $K$ and at least one of $m,n$ even, if $\omega_K(m)=1$ and $\omega_K(n)=2$ or 4, when does $\omega_K(\lcm[m,n])=1$?
\end{question}
Analogous questions can be raised about Theorem \ref{omega-odd-K} (the theorem relating the rank to the factors of Fibonacci numbers for odd $K$) and Theorem \ref{omega-even-K} (the theorem relating the rank to the factors of Fibonacci numbers for even $K$). Theorem \ref{omega-odd-K} holds for even $K$ for moduli that are not multiples of $4$, as proven in Theorem \ref{omega-even-K}. When the modulus is a multiple of 4, it is sometimes false, and sometimes true. Conjecture \ref{even-K-exceptions} speculates about these cases.

\begin{conjecture}\label{even-K-exceptions}
    The following portions of Theorem \ref{omega-odd-K} also apply to even $K$ for all moduli $m > 3$:
    \begin{itemize}
    \item $\omega_K(m)=4$ if and only if $\alpha_K(m)\equiv\pm1\pmod{4}$.\\
    \item If $\omega_K(m)=2$, then $4\mid\pi_K(m)$ and $2\mid\alpha_K(m)$.\\
    \item If $4\not\mid\pi_K(m)$, then $\omega_K(m)=1$.
    \end{itemize}
\end{conjecture}

However, if $4 \mid \pi_K(m)$ and $2 \mid \alpha_K(m)$, $\omega_K(m)$ is not necessarily 2. For example, $\omega_2(8)=1$, and $4 \mid \pi_K(8)=\alpha_K(8)=8$. This also serves as a counterexample to the converse of the third part of Conjecture \ref{even-K-exceptions}, since $\omega_2(8)=1$, but $4 \mid \pi_2(8)$.

\begin{question}
    If $K$ is even and $4 \mid m$, when do the conditions $4 \mid \pi_K(m)$ and $2 \mid \alpha_K(m)$ imply that $\omega_K(m)=2$, and what other values does $\omega_K(m)$ take? Likewise, when does the condition $\omega_K(m)=1$ imply that $4 \not\mid\pi_K(m)$?
\end{question}

\section{Final Remarks}

A general binary recurrence sequence $\mathcal{U}_{n}$ is determined by the four parameters $(a,b,c,d)$ where 
\[
\mathcal{U}_0=c, \qquad \mathcal{U}_1=d, \qquad \mathcal{U}_n=a\times\mathcal{U}_{n-1}+b\times\mathcal{U}_{n-2}.
\]
When $\mathcal{U}_0=0$ and $\mathcal{U}_1=1$, these sequences are often referred to as $(a,b)$-Fibonacci sequences. Unfortunately, the techniques used in this paper will have to be adjusted when $b\neq1$. 
\begin{theorem}[Renault \cite{Renault}]
    In an $(a,b)$-Fibonacci sequence, $\omega(m)\mid2\cdot\text{ord}_m(-b)$.
\end{theorem}
\noindent In $K$-Fibonacci sequences, $b=1$ and $\omega(m)\mid4$. However, $\text{ord}_m(-b)$ takes infinitely many values as $b$ varies. 

There is one particular case where the behavior of $\text{ord}_m(-b)$ is well known, namely if $b=-1$. In this case, $\text{ord}_m(-b)=1$ and $\omega(m)\mid2$. There are five degenerate cases: $a=\pm2, \pm1,0$. These occur when the ratio of the roots of the characteristic polynomial is a root of unity or if $a^2+4b=0$. If $a=1$, then $a^2-4b=-3$, and the sequence is a repeating six-term cycle: $0, 1, 1, 0, -1, -1, \ldots$. For all $m>2$, $\omega(m)_{(1,-1)}=2$, and for $m=2$, $\omega(m)_{(1,-1)}=1$. Similarly, if $a=-1$, then $a^2-4b=-3$, and the sequence is a repeating three-term cycle: $0, 1, -1,\ldots$. For all $m>1$, $\omega_{(-1,-1)}=1$. If instead, $a=2$, then $a^2+4b=0$ and the sequence is exactly: $0, 1, 2, 3, 4, 5, \ldots$, hence for every positive integer $m>1$, $\omega_{(2,-1)}(m)=1$. Somewhat similarly, if $a=-2$, then $a^2+4b=0$ and the sequence is exactly: $0, 1, -2, 3, -4, 5, \ldots$, hence, $\omega_{(-2,-1)}(m)=1$ for all $m>1$. And if $a=0$, then $a^2-4b=-4$, and the sequence is a repeating four-term cycle: $0, 1, 0, -1, \ldots$, hence, $\omega_{(0,-1)}(m)=2$ for all $m>2$ and $\omega_{(0,-1)}(m)=1$ if $m=2$.

\begin{conjecture}\label{negativemult}
The multiplication table in Theorem \ref{K-multiplication} also applies to $(a,b)$-Fibonacci sequences when $b=-1$, in a slightly altered form:
\begin{table}[H]
\centering
\begin{tabular}{cc|cc}
&&&$\omega(m)$\\
&&&\\
&&1&2\\
\hline
&&&\\
&1&1&2\\
&&&\\
$\omega(n)$&2&2&2
\end{tabular}
\caption{Table of $\omega_{(a,-1)}\left(\text{lcm}[m,n]\right)$.}\label{VinsonTableNegative}
\end{table}
\end{conjecture}

There are three other degenerate cases in total for $(a,b)$-Fibonacci sequences, namely when $(a,b)$ equals $(0,1)$ or $(1,0)$ or $(-1,0)$. When $(a,b)=(0,1)$ and when $(a,b)=(1,0)$, the sequence becomes $0, 1, 1, 1, 1, \ldots$. Then, for all $m>1$, $\omega_{(1,0)}=1$. When $(a,b)=(-1,0)$, the sequence becomes $0, 1, 0, -1, 0, 1, 0, -1, \ldots$. Then, for all $m>2$, $\omega_{(-1,0)}=2$. If $m=2$, then $\omega_{(-1,0)}=1$.

Interestingly, the order is sometimes well-behaved for certain other Fibonacci sequences, indicating that more restrictions may be found on the order. For example, the order of the $(3,4)$-Fibonacci sequence modulo $m$ is always 0, 1, or 2 for $m < 20000$. If the sequence is eventually periodic but the period contains no multiples of $m$, some (such as Fiebig, Mbirika, and Spilker \cite{FMS}) define the order to not exist. For illustrative purposes, define the order to be zero if the sequence is eventually periodic modulo $m$ and this period contains no multiples of $m$; this appears to occur if and only if $m$ shares a common factor with either $a$ or $b$. For $m < 20000$, the sets of numbers with order 0 and 2, respectively, were each closed under multiplication (similar to OEIS Conjectures \ref{conj:oeis4} and \ref{conj:oeis1}), while the set of numbers with order 1 was not. However, in the $(3,5)$-Fibonacci sequence, the order takes on 20 distinct values for $m < 100$ (and 112 distinct values for $m < 1000$), indicating that the range of the $(3,5)$-order is likely infinite. Putting everything together, it is possible to speculate about the range of \textit{all} $(a,b)$-Fibonacci sequences.%These observations open the door to further inquiry.

\begin{conjecture}\label{finite-orders}
For an $(a,b)$-Fibonacci sequence, the range of $\omega_{(a,b)}(m)$ for $m>1$ is given by
\begin{enumerate}[label=(\roman*)]
    \item $\{1\}$ if $(a,b) \in \{(0,1),(1,0),(-1,-1),(\pm2,-1)\}$, \\
    \\
    \indent \ \ or if $m=2$ and $(a,b) \in \{(-1,0),(1,-1), (0,-1)\}$,\\
    \item \{2\} if $m>2$ and $(a,b) \in \{(-1,0),(1,-1), (0,-1)\}$,\\
    \item $\{1,2,4\}$ if $b=1$ and $a\neq0$,\\
    \item $\{1,2\}$ if $b=-1$ and $a\not\in\{-2,-1,0,1,2\}$,\\
    \item $\{0,1,2\}$ if $b \neq \pm1$ and $|a|-|b| = 1$,\\
    \item infinite in all other cases.
\end{enumerate}
\end{conjecture}

\begin{conjecture}
    Choose an $(a,b)$-Fibonacci sequence that falls under (vi) of Conjecture \ref{finite-orders}. For all natural numbers $n$, there exists infinitely many positive integers $m$ such that $\omega_{(a,b)}(m)=n$.
\end{conjecture}

%\begin{question}
%    For which recurrence sequences does the order range over finitely many values? What can be said about the behavior of the order in general recurrence sequences?
%\end{question}

\section*{Acknowledgments}
The authors thank aBa Mbirika for thoughtful discussion at the 21\textsuperscript{th} International Fibonacci Conference that contributed to the proof of Theorem \ref{omegaK(2)^x-even}, and to the anonymous referee, whose insightful comments greatly improved the manuscript.

\bibliographystyle{plain}
\bibliography{fiborder}{}

\end{document}